\documentclass{amsart} 

\usepackage[utf8]{inputenc} 
\usepackage{mathtools}
\usepackage{amssymb}
\usepackage{amsmath}
\usepackage{amsthm}
\usepackage{listings}
\usepackage{physics}
\usepackage{calrsfs}
\usepackage{enumerate}
\usepackage{tikz-cd}
\usepackage{mathrsfs}
\usepackage{bm}
\usepackage{esvect}
\usepackage{enumitem}
\usepackage{hyperref}
\usepackage[final]{pdfpages}
\usepackage{verbatim}

\usepackage[backend=biber, style=numeric, sorting=nyt]{biblatex}
 \usepackage[parfill]{parskip}
 
\newtheorem{theorem}{Theorem}[section]
\newtheorem{lemma}[theorem]{Lemma}
\newtheorem{proposition}[theorem]{Proposition}

\theoremstyle{remark}
\newtheorem{remark}[theorem]{Remark}

\numberwithin{equation}{section}

\newcommand{\R}{\mathbb{R}}
\newcommand{\C}{\mathbb{C}}
\newcommand{\Z}{\mathbb{Z}}
\renewcommand{\H}{\mathbb{H}}

\newcommand{\sym}{\mathrm{sym}}
\newcommand{\GH}{\Gamma \backslash \H}
\def\SL{\mathrm{SL}}

\newcommand{\lr}[1]{\left (   {#1} \right )}

\newcommand{\inprod}[2]{\left \langle  {#1} , {#2} \right \rangle}

\usepackage{geometry} 
\usepackage[stretch=10]{microtype}
\usepackage{xcolor}
\definecolor{dark-red}{rgb}{0.4,0.15,0.15}
\definecolor{dark-blue}{rgb}{0.15,0.15,0.4}
\definecolor{medium-blue}{rgb}{0,0,0.5}
\hypersetup{
	pdftitle={Dissipation},
	pdfauthor={Petru Constantinescu},
	pdfnewwindow=true,
	colorlinks, linkcolor={dark-red},
	citecolor={dark-blue}, urlcolor={medium-blue}
}

\title{Dissipation of correlations of holomorphic cusp forms} 
\author[P. Constantinescu]{Petru Constantinescu}
\address{Max Planck Institute for Mathematics, Vivatsgasse 7, 53111 Bonn, 
Germany}

\email{\href{mailto:constantinescu@mpim-bonn.mpg.de}{constantinescu@mpim-bonn.mpg.de}}

\begin{document}

\maketitle

\begin{abstract}
  We obtain a generalisation of the Quantum Unique Ergodicity for holomorphic cusp forms on $\mathrm{SL}_2(\mathbb{Z}) \backslash \mathbb{H}$ in the weight aspect. We show that correlations of masses coming from off-diagonal terms dissipate as the weight tends to infinity. This corresponds to classifying the possible quantum limits along any sequence of Hecke eigenforms of increasing weight.
  
  Our new ingredient is to incorporate the spectral theory of weight $k$ automorphic functions to the method of Holowinsky–-Soundararajan.  For Holowinsky’s shifted convolution sums approach, we need to develop new bounds for the Fourier coefficients of weight $k$ cusp forms. For Soundararajan’s subconvexity approach, we use Ichino’s formula for evaluating triple product integrals.
\end{abstract}

\section{Introduction}

Mass equidistribution of eigenfunctions is a central topic in quantum chaos and number theory. A foremost example is a famous conjecture of Rudnick and Sarnak \cite{RudSar}, which states that normalised Maa{\ss} cusp forms for the modular surface obey Quantum Unique Ergodicity as the eigenvalue tends to infinity. This means that, if $\phi$ is a Maa{\ss} cusp form of eigenvalue $\lambda$, then the measure $\displaystyle \mu_{\phi}:=|\phi(z)|^2 \frac{dxdy}{y^2} $ approaches the uniform distribution measure $\displaystyle \frac{3}{\pi}\frac{dxdy}{y^2}$ as $\lambda \to \infty$. Lindenstrauss \cite{lindenstrauss} showed that for Hecke--Maa{\ss} forms, the only possible limiting measures are of the form $\displaystyle \frac{3}{\pi}c\frac{dxdy}{y^2}$, with $0 <c \leq 1$, and Soundararajan \cite{SoundQUE} completed the proof of Quantum Unique Ergodicity for Hecke--Maa{\ss} forms, showing that $c=1$. 

 

Fix $\Gamma = \SL_2(\Z)$ and $X=\SL_2(\Z) \backslash \H$ the modular curve. We now state the analogue of Quantum Unique Ergodicity for holomorphic Hecke cusp forms, proved by Holowinsky and Soundararajan \cite{HolSound}. 
\begin{theorem}[Holowinsky--Soundararajan] Let $f$ be a holomorphic Hecke cusp form of weight $k$ that is $L^2$-normalised and let $F_k(z)=y^{k/2} f(z)$. Fix any $\phi$ smooth and bounded on $X$. Then we have
\begin{equation*}
    \int_{X}y^k |f(z)|^2  \phi(z) \frac{dx dy}{y^2} \to \frac{3}{\pi} \int_X \phi(z) \frac{dx dy}{y^2} \quad \text{as} \quad k \to \infty;
\end{equation*}
equivalently, this can be rewritten as
\begin{equation*}
    \inprod{\phi F_k}{F_k} \to \frac{1}{\mathrm{vol}(X)} \inprod{\phi}{1} \quad \text{as} \quad k \to \infty.
\end{equation*}
\end{theorem}

This result has inspired important subsequent work. Nelson generalised their results in the level aspect \cite{nelson} and to compact surfaces \cite{Nelson21}, while Lester, Matom\"{a}ki, Radziwi\l \l \cite{LMR18} study the distribution of holomorphic cusp forms at small scales.


In this paper, we obtain a generalisation to off-diagonal terms, where we consider two different eigencusp forms $f$ and $g$ of weights $k_1$ and $k_2$ respectively. We show that correlations dissipate as $\max(k_1,k_2) \to \infty$. We obtain a result about joint distribution of masses in the context of QUE, a subject with interesting recent results, see for example the work of Brooks \cite{Brooks16} on distribution of off-diagonal Eisenstein series $\inprod{\phi E(\cdot, r)}{E (\cdot, r'')}$ or Brooks--Lindenstrauss \cite{ShiLin} on joint quasimodes of the Laplacian.



Let $k$ be an integer. We denote by $\mathcal{A}_k(\Gamma)$ the space of automorphic functions of weight $k$, that is functions $f:\H \to \C$ which transform as
\begin{equation}
    f(\gamma z)= j_{\gamma}(z)^k f(z), \quad \text{for all }  \gamma \in \Gamma,
\end{equation}
where $\displaystyle j_{\gamma}(z)=\frac{cz+d}{|cz+d|}$ with $\gamma=\begin{pmatrix} * & * \\
c & d \end{pmatrix}$. We denote by $\mathcal{L}_k(X)$ the space of automorphic functions of weight $k$ which are square-integrable. We see that if $f \in S_k(\Gamma)$, then $y^{k/2} f(z) \in \mathcal{L}_k(X)$. 

We have the Maa{\ss} raising and lowering operators
\begin{align*}
    K_k: \mathcal{L}_{k}(X) \to \mathcal{L}_{k+2}(X) \quad \text{and} \quad \Lambda_k:\mathcal{L}_{k}(X) \to \mathcal{L}_{k-2}(X) ,
\end{align*}
which allow us to move between automorphic functions of different weights, see \ref{raise and lower} for definitions. Hence, for even integers $k_1 \leq k_2$, we define the operator
\begin{align*}
    R_{k_1}^{k_2}: \mathcal{L}_{k_1}(X) &\to \mathcal{L}_{k_2}(X), &
    \phi & \mapsto \frac{K_{k_2-2}\dots K_{k_1+2}K_{k_1} \phi}{\|K_{k_2-2}\dots K_{k_1+2}K_{k_1} \phi \|},
\end{align*}
where $\| R_{k_1}^{k_2} \phi \|= 1$. We prove the following theorem.

\begin{theorem}
\label{main}
 Fix any $\phi \in C_b(\GH)$ (a bounded function on on $\GH$). Let $f$ and $g$ be $L^2$-normalised  holomorphic Hecke cusp forms of weights $k_1$ and $k_2$ respectively with $k_1 \leq k_2$. 
 Let
 \begin{equation*}
     \delta_{f=g} = \begin{cases}1, & \text{if } f=g; \\ 0, & \text{otherwise}.
     \end{cases}
 \end{equation*}
 Along any sequences of such $f$ and $g$, we have
\begin{equation*}
        \int_{X}  \phi(z) R_{k_1}^{k_2} \lr{y^{k_1/2} f(z)} y^{k_2/2} \overline{g(z)}  \frac{dx dy}{y^2} \to \delta_{f=g} \frac{3}{\pi}  \int_X \phi(z) \frac{dx dy}{y^2} \quad \text{as} \quad k_2 \to \infty.
\end{equation*}
In other words, if $F_{k_1}(z)=y^{k_1/2}f(z)$ and $G_{k_2}(z)=y^{k_2/2}g(z)$, then
\begin{equation*}
    \inprod{\phi \left (R_{k_1}^{k_2} F_{k_1} \right )}{G_{k_2}} \to \delta_{f=g}\frac{1}{\mathrm{vol}(X)} \inprod{\phi}{1} \quad \text{as} \quad k_2 \to \infty.
\end{equation*}
\end{theorem}

\begin{remark}
This corresponds to a generalisation of Quantum Unique Ergodicity by classifying the possible quantum limits of Hecke  cusp forms when we project back to the modular surface. That is, along any sequence of holomorphic Hecke eigenforms of increasing weight, we show there are two possible limit points. Moreover, we obtain explicit rates of cancellations and show that the decay is very fast when $k_2-k_1$ large.
\end{remark}

We also consider the case where we do not raise $F_{k_1}$ to weight $k_2$, but rather project into $\mathcal{L}_{k_2-k_1}(X)$. These statements are not the same, since there are extra normalising factors that play an important role.
\begin{theorem}
\label{fixed difference}
Fix $\phi \in C_b(X)$. Let $l$ be a nonnegative even integer.  Let $f$ and $g$ vary along a sequence of Hecke cusp forms of weights $k$ and $k+l$ respectively. Then 
\begin{equation*}
    \int_{\GH} \left (R_0^{l }\phi(z) \right ) y^{k+l/2} f(z) \overline{g(z)} d \mu(z)  \to \delta_{f=g} \frac{3}{\pi}  \int_{\GH} \phi(z) d\mu(z) \quad \text{as} \quad k\to \infty.
\end{equation*}
In other words, we have
\begin{equation*}
    \inprod{ \left  (R_0^{l }\phi \right ) F_{k} }{G_{k+l}} \to \delta_{f=g} \frac{1}{\mathrm{vol}(X)} \inprod{\phi}{1} \quad \text{as} \quad k \to \infty.
\end{equation*}
\end{theorem}

\begin{remark}
\label{growth of l}
In Theorem \ref{fixed difference}, we can also allow $l$ to grow with $k$. Our method works if $l \leq   c\log \log k$, where $c<\frac{1}{12 \log 2}$.  
\end{remark}

\begin{remark}
It is crucial for us in Theorem \ref{fixed difference} that $\phi$ is obtained from repeated iterations of raising operators. We expect the statement to hold for all $\phi \in \mathcal{L}_l(X)$. However, to achieve this we would also need to compute inner products of the type $\inprod{(R_m^l F_m) G_k}{H_{k+l}}$, which in representation theory corresponds to a triple integral of three discrete series representations. The local factors of such integrals are difficult to estimate. The local factors of triple product integrals where at least one factor comes from principal series representation (Maa{\ss} forms) were computed by Cheng \cite{Cheng21}. In ongoing joint work with Jana, we use a recent breakthrough from \cite{BJN} to evaluate the local factors of discrete triple product integrals and obtain more general results.
\end{remark}

We use the spectral theory of weight $k$ automorphic functions, which we summarise thoroughly in Section \ref{theory weight k}. We can write a decomposition of $\mathcal{L}_k(X)$ in terms of eigenfunctions of the weight $k$ Laplacian $\Delta_k$. The spectral expansion will involve:
\begin{itemize}
    \item  Hecke Maa{\ss} cusp forms $R_0^k u_j$ raised to weight $k$;
    \item raised holomorphic Hecke cusp forms $R_{l}^k(F_l)$, for $0<l \leq k$;
    \item weight $k$ Eisenstein series $E_k \left (z,\frac{1}{2}+it \right )$.
\end{itemize}
Therefore, it is enough to compute inner products of type $\inprod{\phi F_{k_1}}{G_{k_2}}$ or $\inprod{ \phi R_{k_1}^{k_2} F_{k_1}}{G_{k_2}}$, where $\phi$ appears in the spectral decomposition. We proceed similarly as in the work of Holowinsky \cite{Hol} and Soundararajan \cite{Sound}. Our new ingredient is to incorporate the spectral theory of weight $k$ automorphic functions to their method, which we review in Section \ref{theory weight k}. We have two approaches, depending on the size of 
\begin{equation}
    S(f,g):=L(1, \sym^2 f) L(1, \sym^2g) .
\end{equation}
Firstly, we can compute directly the inner products, using Rankin--Selberg unfolding for the Eisenstein series and Ichino's formula for the Maa{\ss} cusp form case and we use computations of Cheng \cite{Cheng21} for the local factors, see Section \ref{identities}. The formulas will involve central values of $L$-functions, to which we apply the weak subconvexity results of Soundararajan. This will win if $S(f,g)$ is large.

Alternatively, we can expand the inner products in terms of the Fourier expansions. We need bounds for the Fourier coefficients of weight $k$ automorphic forms, which we compute in Section \ref{Fourier bounds}. This approach boils down to bounding shifted convolution sums, where we apply the results of Holowinsky, see Section \ref{holowinsky}. This will win if $S(f,g)$ is sufficiently small. In both approaches it is crucial that holomorphic cusp forms obey the Ramanujan--Petersson conjecture. We put everything together and complete the proofs of Theorems \ref{main} and \ref{fixed difference} in Section \ref{proofs}.


\section*{Acknowledgements}
We would like to thank Peter Humphries for suggesting the problem to us, for many useful conversations and providing help with the reference \cite{Cheng21}. We would like to thank Yiannis Petridis and Valentin Blomer for their encouragement and feedback. This work represents part of the author’s doctoral dissertation
written at University College London. The author is grateful to Max Planck Institute for Mathematics in Bonn for its hospitality and financial support.

\section{Notation}
\label{notation}

We begin by reviewing some properties of $L$-functions. Our main references for this section are \cite[Chapter 5]{IK} and \cite[Chapter 1]{bump}. Let $L(s,f)$ be the Dirichlet series with an Euler product of degree $d$
\begin{equation*}
    L(s,f)=\sum_{n \geq 1} \frac{a_f(n)}{n^s}=\prod_p \prod_{j=1}^d\lr{1-\frac{\alpha_j(p)}{p^s}}^{-1},
\end{equation*}
which is absolutely convergent for $\Re(s)>1$. We write
\begin{equation*}
    L_{\infty}(s,f)=N^{s/2} \prod_{j=1}^m \Gamma_{\R}(s+\mu_j) ,
\end{equation*}
where $\Gamma_{\R}(s)=\pi^{-s/2} \Gamma(s/2)$, $N$ denotes the conductor and $\mu_j \in \C$ are some parameters. We assume we a completed $L$-function with functional equation
\begin{equation*}
    \Lambda(s,f)=L_{\infty}(s,f) L(s,f) =\kappa \Lambda( 1-s,f),
\end{equation*}
where $\kappa$ is a complex number of magnitude 1 (the root number).

If $L(s,f)$ is a $L$-function with this properties, then we define the analytic conductor to be
\begin{equation}
    C(f)=N \prod_{j=1}^d (1+ |\mu_j|).
\end{equation}
Similarly, we define
\begin{equation}
    C(s,f)= N \prod_{j=1}^d (1+ |\mu_j +s|).
\end{equation}
We want to make use of the following result of Soundararajan in \cite{Sound}. In its rough form, if $L(s,f)$ is an $L$-function with the properties above, and additionally satisfies Ramanujan conjectures ($|\alpha_j(p)| \leq 1$, for $1 \leq j \leq d$ and all primes $p$) and $\Re(\mu_j)>-1$, then the following weak-subconvexity bound holds:
\begin{equation*}
    \label{subconvex}
    L\lr{\frac{1}{2}, f} \ll_{\epsilon} \frac{C(f)^{1/4}}{(\log(C(f))^{1-\epsilon}} .
\end{equation*}

Now let $f$ and $g$ be holomorphic Hecke eigenforms of weights $k_1$ and $k_2$ respectively. We assume they are $L^2$-normalised:
\begin{equation*}
\int_{\Gamma \backslash \H} y^{k_1} |f(z)|^2 d \mu(z) =  \int_{\Gamma \backslash \H} y^{k_2} |g(z)|^2 d \mu(z) =1. 
\end{equation*}
Also we denote $F_{k_1}(z)=y^{k_1/2} f(z)$ and $F_{k_2}(z)=y^{k_2/2} f(z)$. We write
\begin{equation*}
     f(z)=\sum_{n \geq 1} a_f(n) e(nz) , \quad g(z)=\sum_{n \geq 1} a_g(n) e(nz),
\end{equation*}
where
\begin{equation*}
    a_f(n)=\lambda_f(n) a_f(1) n^{(k_1 -1)/2}, \quad  a_g(n)=\lambda_g(n) a_g(1) n^{(k_2 -1)/2}.
\end{equation*}
In this form, $\lambda_f(n)$ and $\lambda_g(n)$ are the eigenvalues of the Hecke operators $T_n$. We define
\begin{equation*}
    L(s,f)=\sum_{n \geq 1} \frac{\lambda_f(n)}{n^s}= \prod_p \lr{1-\frac{\alpha_f(p)}{p^s}}^{-1} \lr{1-\frac{\beta_f(p)}{p^s}}^{-1}
\end{equation*}
and similarly for $L(g,s)$. We know that $|\alpha_f(p)|=|\beta_f(p)|=|\alpha_g(p)|=|\beta_g(p)|=1$, for all primes $p$, so the Ramanujan conjecture holds, by the work of Deligne. By definition, we have the factorisation of Hecke polynomials
\begin{align*}
    1-\lambda_f(p) p^{-s} + p ^{-2s} &= \lr{1- \alpha_f(p) p^{-s}} \lr{1- \beta_f(p) p^{-s}}, \\
    1-\lambda_g(p) p^{-s} + p ^{-2s} &= \lr{1- \alpha_g(p) p^{-s}} \lr{1- \beta_g(p) p^{-s}}.
\end{align*}

The gamma factors of $L(s,f) $ are given by
\begin{equation*}
    L_{\infty}(s,f)=\Gamma_{\R}(s+(k_1-1)/2)\Gamma_{\R}(s+(k_1+1)/2).
\end{equation*}
This implies that
\begin{equation*}
    C(f)= \frac{k_1+1}{2}\frac{k_1+3}{2}\asymp k_1^2,
\end{equation*}
and similarly $C(g) \asymp k_2^2$.

Now we define the Rankin--Selberg convolution $L$-function $L(s, f \times g)$ as
\begin{equation*}
    L(f \times g, s) :=\prod_p \lr{1-\frac{\alpha_f(p) \alpha_g(p)}{ p^s}}^{-1} \lr{1-\frac{\alpha_f(p) \beta_g(p)}{ p^s}}^{-1} \lr{1-\frac{\beta_f(p) \alpha_g(p)}{ p^s}}^{-1}\lr{1-\frac{\beta_f(p) \beta_g(p)}{ p^s}}^{-1}  .
\end{equation*}
It admits analytic continuation to all $s \in \C$ and it has a simple pole at $s=1$ if and only if $f=g$. Assume by symmetry that $k_1 \leq k_2$. The Gamma factors are  
\begin{equation}
    \label{gamma1}
    L_{\infty}(s, f \times g)= \Gamma_{\R}\lr{s+\frac{k_1+k_2}{2} }\Gamma_{\R}\lr{s+\frac{k_1+k_2}{2}-1 } \Gamma_{\R}\lr{s+\frac{k_2-k_1}{2} } \Gamma_{\R}\lr{s+\frac{k_2-k_1}{2} +1 } .
\end{equation}
This implies that
\begin{equation}
    \label{C1}
    C(f \times g) \asymp (k_1+k_2)^2(1+k_2-k_1)^2.
\end{equation}

When $f=g$, we define the symmetric square $L$-function
\begin{align*}
    L(s, \sym^2f):= \prod_p \lr{ 1 - \frac{\alpha_f(p)^2}{p^s}}^{-1} \lr{ 1- \frac{\beta_f(p)^2}{p^s}}^{-1} \lr{1-\frac{1}{p^s}}^{-1} 
    &=\frac{1}{\zeta(s)} L(s, f \times f)\ .
\end{align*}

We can write the first Fourier coefficient $a_f(1)$ as
\begin{equation}
\label{af1}
    |a_f(1)|^2= \frac{ 2 \pi^2 (4 \pi)^{k_1-1}  }{  \Gamma (k_1) L(\sym^2f, 1)} \ .
\end{equation}

\section{Spectral theory of weight $k$ automorphic forms}
\label{theory weight k}

We quote \cite[Chapter 4]{DFI02}, \cite[Chapter 2]{bump} for detailed expositions on the analytical theory of weight $k$ automorphic forms. Let $k$ be an integer. We denote by $\mathcal{A}_k(\Gamma)$ the space of automorphic functions of weight $k$, that is functions $f:\H \to \C$ that transform by
\begin{equation}
    f(\gamma z)= j_{\gamma}(z)^k f(z), \quad \text{for all }  \gamma \in \Gamma,
\end{equation}
where $\displaystyle j_{\gamma}(z)=\frac{cz+d}{|cz+d|}$ with $\gamma=\begin{pmatrix} * & * \\
c & d \end{pmatrix}$. Note that we have the cocycle relation
\begin{equation*}
    j_{\gamma_1 \gamma_2}(z)=j_{\gamma_1}(\gamma_2 z) j_{\gamma_2}(z), \quad \text{for all } \gamma_1, \gamma_2 \in \Gamma.
\end{equation*}
Let  $\mathcal{L}_k(\Gamma)$ the automorphic functions of weight $k$ that are square-integrable. On $\mathcal{L}_k(\Gamma)$ we define the inner product
\begin{equation}
    \inprod{f}{g}=\int_{\Gamma \backslash \H} f(z) \overline{g(z)} d \mu .
\end{equation}

We consider the Maa{\ss} raising and lowering operators acting on $C^{\infty}(\H)$ (smooth functions on $\H$)
\begin{align}
\label{raise and lower}
\begin{split}
    K_k=\frac{k}{2}+y \left ( i \pdv{}{x} + \pdv{}{y}\right ) = \frac{k}{2} + (z-\overline{z})\pdv{}{z} ,\\
    \Lambda_k=\frac{k}{2}+y \left ( i \pdv{}{x} - \pdv{}{y}\right ) = \frac{k}{2} + (z-\overline{z})\pdv{}{\overline{z}} .
\end{split}
\end{align}
These operators are used to map between spaces of different weights:
\begin{align*}
    K_k : C^{\infty}(\Gamma) \cap \mathcal{L}_k(\Gamma) \to C^{\infty}(\Gamma) \cap \mathcal{L}_{k+2}(\Gamma), \\ \Lambda_k : C^{\infty}(\Gamma) \cap \mathcal{L}_{k}(\Gamma) \to C^{\infty}(\Gamma) \cap \mathcal{L}_{k-2}(\Gamma), 
\end{align*}
and satisfy the following property:
\begin{align}
\label{adjoint}
    \inprod{K_k f}{g}=-\inprod{f}{\Lambda_{k+2}g},
\end{align}
for $f \in C^{\infty}(\Gamma) \cap \mathcal{L}_k(\Gamma)$ and $g \in C^{\infty}(\Gamma) \cap \mathcal{L}_{k+2}(\Gamma)$. Moreover, the following product rule holds:
\begin{align}
\label{product}
\begin{split}
    K_{k+l}(g_k g_l) &= (K_k g_k)g_l + g_k (K_l g_l), \\
    \Lambda_{k+l}(g_k g_l) &= (\Lambda_k g_k) g_l + g_k (\Lambda_l g_l),
\end{split}
\end{align}
where $g_k$ and $g_l$ are smooth automorphic functions of weights $k$ and $l$ respectively.

The Laplace operator of weight $k$ is defined by
\begin{equation*}
    \Delta_k=y^2 \left ( \pdv[2]{}{x} + \pdv[2]{}{y} \right ) - i k \pdv{}{x} .
\end{equation*}
This can be written in terms of the raising and lowering operators as
\begin{equation}
    \Delta_k=-K_{k-2} \Lambda_k - \lambda(k/2) = -\Lambda_{k+2} K_k - \lambda (-k/2),
\end{equation}
where
\begin{equation}
    \lambda(s):=s(1-s).
\end{equation}
The operator $\Delta_k$ acts on $\mathcal{A}_k(\Gamma) \cap C^{\infty}(\Gamma)$. We define a {\it Maa{\ss} form} to be a smooth automorphic function of weight $k$ which is an eigenfunction of $\Delta_k$. Let $\mathcal{A}_k(\Gamma,s)$ denote the space of Maa{\ss} forms with eigenvalue $\lambda(s)$. We also note that, if $f(z) \in \mathcal{A}_k(\Gamma,s)$ has at most polynomial growth in cusp, it has a Fourier expansion of the form
\begin{equation*}
    f(z)=a_0(y) + \sum_{n \neq 0} a_f(n) W_{\frac{kn}{2|n|} , s-\frac{1}{2}}(4 \pi |n| y) e (nx),
\end{equation*}
where $W_{\alpha, \beta}(z)$ is the Whittaker function, see \cite{DFI02} for more details. 

We denote by $\mathcal{B}_k(\Gamma)$ the space of smooth automorphic functions of weight $k$ such that $f, \Delta_k f \in \mathcal{L}_k(\Gamma)$. Then $-\Delta_k$ defines a symmetric, non-negative operator on $\mathcal{B}_k(\Gamma)$. The space $\mathcal{B}_k(\Gamma)$ is dense in $\mathcal{L}_k(\Gamma)$ and the operator $-\Delta_k$ admits a self-adjoint extension to $\mathcal{L}_k(\Gamma)$ and we can study the spectral decomposition of this space.

\subsection{Eisenstein Series} The Eisenstein series of weight $k$ is defined by
\begin{equation}
\label{eisenstein k}
    E_k(z,s):= \sum_{\gamma \in \Gamma_{\infty} \backslash \Gamma}(\Im \gamma z)^s j_{\gamma}(z)^{-k},
\end{equation}
The series \eqref{eisenstein k} converges absolutely for $\Re(s)>1$ and has analytic continuation to the whole complex plane. Unless $k=0$, $E_k(z,s)$ has no poles for $\Re(s) \geq 1/2$. If $k=0$, then $E(z,s)$ has a pole at $s=1$ with residue 
\begin{equation}
\label{residue_eisenstein}
    \Res_{s=1}E(z,s)= \frac{3}{\pi}.
\end{equation}
If $s$ is not a pole of $E_k(z,s)$, then $E_k(z,s)$ is a weight $k$ Maa{\ss} form with eigenvalue $\lambda(s)$, but it is not in $\mathcal{L}_k(\Gamma)$.
We note that
\begin{align*}
    K_k E_k(z,s) = \left ( \frac{k}{2} + s \right ) E_{k+2}(z,s), \quad
    \Lambda_k E_k(z,s) = \left ( \frac{k}{2} - s \right ) E_{k-2}(z,s).
\end{align*}
Hence, if $k$ is an even positive integer,
\begin{equation}
\label{eisenstein relation}
    K_{k-2}\dots K_2 K_0 E(z,s)=s(s+1) \dots (s+k/2-1) E_{k}(z,s)=\frac{\Gamma(s+k/2)}{\Gamma(s)}E_{k}(z,s).
\end{equation}

As in \cite{Ja94}, \cite{DFI02}, \cite{Du88} or \cite{O18}, the Fourier expansion of $E_k(z,s)$ is given by
\begin{align}
    \label{eisenstein expansion}
    \begin{split}
    E_{k}(z, s) =& y^{s} + \frac{(-1)^{k/2} \Gamma \lr{s}^2}{\Gamma \lr{s-\frac{k}{2}}\Gamma \lr{s+\frac{k}{2}}} \phi(s) y^{1- s} \\
    &+ \frac{(-1)^{k/2} \Gamma \lr{s}}{2 \Gamma \lr{s+\frac{|k|}{2}} \xi(2s)} \sum_{n>0} |n|^{s-1} \sigma_{1-2s}(|n|)W_{|k|/2, s-1/2}(4 \pi |n| y) e (nx) \\
    &+ \frac{(-1)^{k/2} \Gamma \lr{s}}{2 \Gamma \lr{s -\frac{|k|}{2}} \xi(2s)} \sum_{n<0} |n|^{s-1} \sigma_{1-2s}(|n|)W_{-|k|/2, s-1/2}(4 \pi |n| y) e (nx),
    \end{split}
\end{align}
where
\begin{align*}
    \phi(s) &= \frac{\xi(2s-1)}{\xi(2s)}, \\
    \xi(s) &= \pi^{-s/2} \Gamma(s/2) \zeta(s) = \Gamma_{\R}(s) \zeta(s), \\
    \sigma_\nu (n) &= \sum_{d \mid n} d^{\nu} . 
\end{align*}

Let $\psi(y)$ be a smooth compactly supported function on $\R^+$. Then we define the incomplete Eisenstein series
\begin{equation*}
    E_k(z |\psi):= \sum_{\gamma \in \Gamma_{\infty} \backslash \Gamma} \psi (\Im \gamma z) j_{\gamma}(z)^{-k}, 
\end{equation*}
that is in $\mathcal{L}_k(\Gamma)$, but it is not a Maa{\ss} form. We denote by $\mathcal{E}_k(\Gamma)$ the space of all incomplete Eisenstein series. Then $\Delta_k$ acts on $\mathcal{E}_k(\Gamma)$ with purely continuous spectrum which covers the interval $[1/4, \infty)$ with multiplicity one. Moreover, for any $f \in \mathcal{E}_k(\Gamma) $, we have the expansion
\begin{equation*}
    f(z)=\frac{1}{4 \pi} \int_{-\infty}^{\infty} \inprod{f}{E_k \left (\cdot, \frac{1}{2}+it \right )}E_k \left (\cdot, \frac{1}{2}+it \right ) dt .
\end{equation*}

We let 
\begin{equation*}
    \Psi(s):= \int_0^{\infty} \psi(y) y^s \frac{dy}{y}
\end{equation*}
be the Mellin transform of $\psi$. Hence, $\Psi(s)$ is entire and satisfies
\begin{equation}
\label{decay}
    \Psi(s) \ll (1+|s|)^{-A}
\end{equation}
for any $A >0$, uniformly in vertical strips. By the Mellin inversion theorem, we have
\begin{equation*}
    \psi(y) = \frac{1}{2 \pi i } \int_{(\sigma)} y^{-s} \Psi(s) ds 
\end{equation*}
for $\sigma>1$. Using this, we observe that
\begin{equation}
    \label{blabla}
    E_{k}(z | \psi) = \frac{1}{2 \pi i } \int_{(2)} \Psi(-s) E_k(z,s) ds .
\end{equation}

\subsection{Cusp forms} The orthogonal complement of $\mathcal{E}_k(\Gamma)$ in $\mathcal{L}_k(\Gamma)$ consists of functions whose zero Fourier coefficient vanishes, which we denote by $\mathcal{C}_k(\Gamma)$. Then $\Delta_k$ acts on $\mathcal{C}_k(\Gamma)$ with purely discrete spectrum. We now provide a description of this space. 

Let $\mathcal{C}_k(\Gamma,s)$ be the space of Maa{\ss} cusp forms of weight $k$ and eigenvalue $\lambda(s)$. Then $K_k : \mathcal{C}_k(\Gamma,s) \to \mathcal{C}_{k+2} (\Gamma, s) $ and $\Lambda_k: \mathcal{C}_{k} (\Gamma, s) \to \mathcal{C}_{k-2} (\Gamma, s)$. Also,
\begin{align*}
    K_k F = 0 \iff \lambda(s)=\lambda(-k/2) \iff y^{k/2} \overline{f(z)} \text{ is holomorphic in } z, \\ 
     \Lambda_k F = 0 \iff \lambda(s)=\lambda(k/2) \iff y^{-k/2} f(z) \text{ is holomorphic in } z.
\end{align*}
If $\lambda(s)\neq \lambda(-k/2)$, then the map
\begin{equation*}
    \lr{\lambda(s) -\lambda \lr{-\frac{k}{2}}}^{-1/2} K_k: \mathcal{C}_{k} (\Gamma, s) \to \mathcal{C}_{k+2} (\Gamma, s)
\end{equation*}
is a bijective isometry. A similar statement holds for $\Lambda_k$. Now for even integers $k_1 < k_2$ and $\lambda(s) \not \in \{ \lambda(-k_1/2), \dots \lambda (-k_2/2+1)\}$, we define the bijective isometry $R_{k_1}^{k_2} : \mathcal{C}_{k_1} (\Gamma, s) \to \mathcal{C}_{k_2} (\Gamma, s)$ given by
\begin{equation}
    \label{operation}
    R_{k_1}^{k_2} (s):= \prod_{\substack{k_1 \leq l < k_2 \\ l \equiv 2 \text{ mod }2}} \lr{\lambda(s) -\lambda \lr{-\frac{k}{2}}}^{-1/2}  K_{k_2-2} \dots K_{k_1+2}K_{k_1} .
\end{equation}
When $k \geq 0$, the eigenspace of $\Delta_k$ with eigenvalue $\lambda(k/2)$ is given by
\begin{equation}
    \mathcal{C}_{k} \left (\Gamma, \frac{k}{2} \right ) = \left \{ y^{k/2} f(z) \mid f\in S_k(\Gamma)\right \}
\end{equation}
and
\begin{equation}
    \mathcal{C}_{-k} \left (\Gamma, \frac{k}{2} \right ) = \left \{ y^{k/2} \overline{f(z)} \mid f\in S_k(\Gamma)\right \} .
\end{equation}
The eigenspaces of $\Delta_k$ in $\mathcal{C}_{k} \left (\Gamma, m/2 \right )$ for even $m$ in the range $0<m \leq k$ are determined by classical cusp forms in $S_m(\Gamma)$ with repeated applications of the Maa{\ss} raising operators.

Putting everything together, we have the following theorem, see \cite[Corollary 4.4]{DFI02}.

\begin{theorem}
Let $k$ be an even positive integer. Let $\{ u_j (z)\}$ be an orthonormal basis of Maa{\ss} cusp forms of $\mathcal{C}_0(\Gamma)$ with corresponding eigenvalues $\lambda(s_j)$. Also, choose $\{f_{j,m}\}$ an orthonormal basis for $S_m(\Gamma)$. Then an orthonormal basis of $\mathcal{C}_k(\Gamma)$ is given by
\begin{align*}
    u_{j,k}(z) &:= \prod_{0\leq l<k/2} (\lambda(s_j)-\lambda(-l))^{-1/2}K_{2l}(u_j(z)), \\
    u_{j,m,k}(z) &:=\prod_{m \leq l <k/2} \left ( \lambda (m) - \lambda (-l) \right )^{-1/2} K_{2l}\lr{y^m f_{j, 2m}(z)} .
\end{align*}
\end{theorem}

\begin{remark}
Since Selberg's eigenvalue conjecture holds for $\Gamma=\mathrm{SL}_2(\Z)$, all the points $s_j$ are on the line $\Re(s_j)=1/2$.
\end{remark}

\begin{remark}
We choose an orthonormal basis of $\mathcal{C}_k(\Gamma)$ consisting of Hecke--Maa{\ss} cusp forms, i.e. common eigenfunctions of the Laplacian $\Delta_k$ and all Hecke operators $T_n$. This is possible because the operators $T_n$ commute with $\Delta_k$, $K_k$ and $\Lambda_k$. We denote such a basis by
\begin{equation}
    B_k:=\{ u_{j,k}\}\cup \lr{ \bigcup_{0<m \leq k/2} \{ u_{j,m,k}\} }.
\end{equation}
\end{remark}
We can compute the normalisation factors, as in \cite[p. 508]{DFI02}. They are given by
\begin{align}
    \alpha^2(s,k) &:=  \prod_{0\leq l<k/2} (\lambda(s_j)-\lambda(-l))^{-1} = (-1)^{k/2}\frac{\Gamma(s-k/2)}{\Gamma(s+k/2)} , \label{alpha} \\
    \beta^2(m,k) &:= \prod_{m/2 \leq l <k/2} \left ( \lambda (m) - \lambda (-l) \right )^{-1}=\frac{\Gamma(m)}{\Gamma \lr{\frac{k+m}{2}} \Gamma 
    \lr{ \frac{k-m}{2} +1}}. \label{beta}
\end{align}

If $f\in S_{k_1}(\Gamma)$ and $F_{k_1}=y^{k_1/2} f(z) \in \mathcal{C}_{k_1}(\Gamma, k_1/2)$, we just denote the isometry $R_{k_1}^{k_2}(k_1/2)$ from \eqref{operation} by $R_{k_1}^{k_2}: \mathcal{C}_{k_1}(\Gamma, k_1/2) \to \mathcal{C}_{k_2}(\Gamma, k_1/2) $ given by
\begin{equation}
\label{operator}
    R_{k_1}^{k_2} F_{k_1} = \beta(k_1, k_2) K_{k_2-2}..\dots K_{k_1} F_{k_1}.
\end{equation}

If $u_j$ is a cuspidal Maa{\ss} form with eigenvalue $\lambda(1/2 + i t_j)$, then its Fourier expansion is given by
\begin{equation*}
    u_j(z) = \sum_{n \neq 0} \frac{c_j(|n|)}{\sqrt{|n|}} W_{0, i t_j} (4 \pi |n| y) e (nx).
\end{equation*}
If $u_j$ is a Hecke eigenform, then the Hecke eigenvalues are given by $c_j(n)/c_j(1)$, for positive $n$. We can relate it to the Fourier expansion of $u_{j,k}$, as in \cite{Ja94}:
\begin{align}
\label{fourier maass cusp k}
    \begin{split}
        u_{j,k}(z)= & \frac{(-1)^{k/2} \Gamma(1/2+ it_j)}{\Gamma \lr{\frac{1}{2}+ \frac{k}{2}+ i t_j}} \sum_{n>0}  \frac{c_j(|n|)}{\sqrt{|n|}} W_{k/2, i t_j} (4 \pi |n| y) e (nx) \\
        &+  \frac{(-1)^{k/2} \Gamma(1/2+ it_j)}{\Gamma \lr{\frac{1}{2}- \frac{k}{2}+ i t_j}} \sum_{n<0}  \frac{c_j(|n|)}{\sqrt{|n|}} W_{-k/2, i t_j} (4 \pi |n| y) e (nx) .
    \end{split}
\end{align}
Now, if $f(z) \in S_{k_1}(\Gamma)$ has Fourier expansion
\begin{equation*}
    f(z)= a_f(1) \sum_{n=1}^{\infty} \lambda_f(n) n^{\frac{k_1-1}{2}} e(nz),
\end{equation*}
then we have the expansion
\begin{equation}
\label{fourier raised cusp}
    R_{k_1}^{k_2}(F_{k_1}(z)) = (-1)^{\frac{k_2-k_1}{2}} \beta(k_1,k_2) a_f(1) \sum_{n=1}^{\infty} \frac{\lambda_f(n)}{\sqrt{n}}W_{\frac{k_2}{2}, \frac{k_1-1}{2}}(4 \pi n y) e (nx),
\end{equation}
where $F_{k_1}(z)=y^{k_1/2} f(z)$ as above.

\section{Integral triple product identities}
\label{identities}

Fix $f$ and $g$ holomorphic cusp forms of weights $k_1$ and $k_2$ respectively with $k_1 \leq k_2$. Denote $F_{k_1}=y^{k_1/2} f(z)$ and $G_{k_2}(z)=y^{k_2/2} g(z)$. In this section we evaluate the inner products $\inprod{\phi F_{k_1}}{G_{k_2}}$, where $\phi$ is an automorphic form of weight $k_2-k_1$. If $\phi$ is an Eisenstein series, we use the classical Rankin--Selberg integral method. If $\phi$ is a cusp form, we evaluate the triple product integral using Ichino's formula \cite{Ichino08}. In both cases, it boils down to estimating central values $L(f \times g, 1/2)$ or $L(\phi \times f \times g, 1/2)$, to which we apply the  subconvexity bounds of Soundararajan from \cite{Sound}.

We begin with the following proposition, which uses the Rankin--Selberg unfolding, see \cite[Proposition 13.1]{Iw2}.
\begin{proposition}
\label{noname}
We have
\begin{align*}
  (4 \pi)^{1-s-\frac{k_1+k_2}{2}} \Gamma \lr{s+\frac{k_1+k_2}{2}-1}   a_f(1) \overline{a_g(1)}\frac{L(f \times g, s)}{\zeta(2s)} =    \int_X y^{(k_1+k_2)/2} f(z) \overline{g(z)} E_{k_2-k_1}(z,s) d \mu  .
\end{align*}
\end{proposition}
\begin{proof}
Using an unfolding argument, for $\Re(s)>1$ we write the integral as
\begin{align*}
    \int_{\Gamma \backslash \H} F_{k_1}(z) \overline{G_{k_2}(z)}   & E_{k_2-k_1} (z)   d\mu = \int_{\Gamma \backslash \H} F_{k_1}(z) \overline{G_{k_2}(z)} \sum_{\gamma \in \Gamma_{\infty} \backslash \Gamma}(\Im \gamma z)^s j_{\gamma}(z)^{-(k_2-k_1)} d\mu \\
    &=  \sum_{\gamma \in \Gamma_{\infty} \backslash \Gamma} \int_{\Gamma \backslash \H} (\Im \gamma z)^s F_{k_1}(z) \overline{G_{k_2}(z)} j_{\gamma}(z)^{k_1-k_2} d\mu \\&= \int_{0}^1 \int_0^\infty y^s y^{\frac{k_1+k_2}{2}}  f(z) \overline{g(z)} dy dx \\ &= \int_{0}^1 \int_0^\infty  y^{s+\frac{k_1+k_2}{2}}  \sum_{n,m \geq 1} a_f(n) \overline{a_g(n)} e^{2 \pi i (n-m) x} e^{-2 \pi (n+m) y} dy dx \\&= a_f(1) a_g(1) \sum_{n \geq 1}  \lambda_f(n) \lambda_g(n) n^{\frac{k_1+k_2}{2}-1} \int_{0}^{\infty} y^{s+\frac{k_1+k_2}{2}} e^{-4 \pi n y} dy \\
    &=( 4 \pi) ^{1-s-\frac{k_1+k_2}{2}} \Gamma \lr{s+\frac{k_1+k_2}{2}-1}   a_f(1) \overline{a_g(1)} \sum_{n \geq 1} \lambda_f(n) \lambda_g(n) n^{-s} .
\end{align*}
\end{proof}

We now write the inner products involving the Eisenstein series.
\begin{lemma}
\label{eisenstion inner products}
Let $s=\frac{1}{2}+it$ and $\alpha=\frac{k_2-k_1}{2}$. Then
\begin{equation*}
     \left |\inprod{E_{k_2-k_1} \lr{\cdot,  s} F_{k_1}}{G_{k_2}} \right | \ll_{\epsilon} \frac{(1+|t|)^{3/2}(\log k_2)^{-1+\epsilon} (1+\alpha)^{1/2}}{ (L(1, \sym^2 f) L(1, \sym^2 g))^{1/2}}.
\end{equation*}
and, for $k_1 < k_2$,
\begin{equation*}
    \left |\inprod{E \lr{\cdot,  s} R_{k_1}^{k_2} F_{k_1}}{G_{k_2}} \right | \ll_{\epsilon} \frac{\Gamma\lr{k_2-\alpha }^{1/2} }{\Gamma(k_2)^{1/2}  \Gamma\lr{\alpha}^{1/2}} k_2^{{\epsilon}} (1+|t|)^{3/2}\left | s (s+1) \dots \lr{s+\alpha-1}\right | .
\end{equation*}
\end{lemma}

\begin{proof}
We use Proposition \ref{noname} and \eqref{af1} to obtain
\begin{align*}
    \left | \inprod{E_{k_2-k_1} \lr{\frac{1}{2}+it, \cdot} F_{k_1}}{G_{k_2}} \right | = \frac{ \pi^{3/2}\Gamma \lr{\frac{k_2+k_1}{2} - \frac{1}{2} + it} L \lr{\frac{1}{2} + it, f \times g }}{ \zeta(1+2 i t)  \Gamma(k_1)^{1/2} \Gamma(k_2)^{1/2} L(1, \sym^2 f)^{1/2} L(1, \sym^2 g)^{1/2}} .
\end{align*}
We use the weak subconvexity bound of Soundararajan \cite{Sound}:
\begin{equation*}
    \left | L \lr{\frac{1}{2}+it, f \times g}\right | \ll \frac{(k_1+k_2)^{1/2}(1+k_2-k_1)^{1/2}}{(\log k_2)^{1-\epsilon}} (1+|t|).
\end{equation*}
We now use that for $\sigma>0$, $|\Gamma(\sigma +it)| \leq \Gamma(\sigma)$ and employ Stirling formula to deduce that $\Gamma (x +1/2) \sim \Gamma (x) \sqrt{x}$ as $x \to \infty$.  Since $|\zeta(1+it)| \gg 1/\log(1+|t|)$, we obtain 
\begin{equation}
    \label{inner product 1}
     \left | \inprod{E_{k_2-k_1} \lr{\frac{1}{2}+it, \cdot} F_{k_1}}{G_{k_2}} \right | \ll_{\epsilon} \frac{\Gamma \lr{\frac{k_2+k_1}{2}} (1+k_2-k_1)^{\frac{1}{2}} (1+|t|)^{1+\epsilon} } { (\log k_2)^{1-\epsilon}  \Gamma(k_1)^{1/2} \Gamma(k_2)^{1/2} L(1, \sym^2 f)^{1/2} L(1, \sym^2 g)^{1/2}}
\end{equation}
Finally we see that
\begin{equation*}
    \frac{\Gamma\lr{\frac{k_1+k_2}{2}}}{\Gamma(k_1)^{1/2} \Gamma(k_2)^{1/2}} = \frac{{k_1 +k_2 -2 \choose k_1-1}^{1/2}}{{k_1 +k_2 -2 \choose (k_1+k_2)/2-1}^{1/2}} \leq 1.
\end{equation*}
and the first part follows.

Now, for the second part, we use the adjointness property \eqref{adjoint}, the product rule \eqref{product}, together with the fact that $\Lambda_{k_2}G_{k_2}(z)=0$, to see that
\begin{align*}
     & \inprod{E \lr{\frac{1}{2}+it, z}  R_{k_1}^{k_2} F_{k_1}(z)}{G_{k_2}(z)} \\ =& \beta(k_1, k_2) \inprod{E \lr{\frac{1}{2}+it, z} (K_{k_2-2}\dots K_{k_1} F_{k_1}(z)) }{G_{k_2}(z)} \\
    =& (-1)^{\frac{k_2-k_1}{2}}\beta(k_1, k_2) \inprod{\lr{K_{k_2-k_1-2} \dots K_0 E \lr{\frac{1}{2}+it, z}} F_{k_1}(z)}{G_{k_2}(z)} \\
    =& (-1)^{\frac{k_2-k_1}{2}} \frac{\beta(k_1, k_2) \Gamma \lr{\frac{k_2-k_1}{2} + \frac{1}{2} + it}}{\Gamma \lr{\frac{1}{2} + it}} \inprod{E_{k_2-k_1}\lr{\frac{1}{2}+it, z} F_{k_1}(z)}{G_{k_2}(z)} .
\end{align*}

If $k_1=k_2$, then the conclusion follows.  Now assume $k_1 < k_2$. Substituting $\beta(k_1, k_2)$ from \eqref{beta} and using \eqref{inner product 1}, we have that $\left | \inprod{E \lr{\frac{1}{2}+it, \cdot} R_{k_1}^{k_2} F_{k_1}}{G_{k_2}} \right |$ is bounded by
\begin{align*}
     \frac{\Gamma\lr{\frac{k_1+k_2}{2} }^{1/2} }{\Gamma(k_2)^{1/2}  \Gamma\lr{\frac{k_2-k_1}{2}}^{1/2}} \left | \prod_{0 \leq j < \frac{k_2-k_1}{2}} \lr{\frac{1}{2}+j +it}  \right | \frac{(1+|t|)^{1+\epsilon}}{(\log k_2)^{1- \epsilon} S(f,g)^{\frac{1}{2}}}.
\end{align*}
We use the bound $L(1, \sym^2 f) \gg (\log k_1)^{-1}$, see \cite{HoLo94}, and similarly for $g$. Hence the contribution from the last fraction is bounded by $k_2^{\epsilon}$ and the conclusion follows.




\end{proof}

Next, we evaluate the inner products involving Hecke--Maa{\ss} cusp forms.
\begin{lemma}
\label{inner product cusp}
Let $\epsilon>0$. We have 
\begin{equation*}
      \left | \inprod{u_{j,k_2-k_1} F_{k_1}}{G_{k_2}} \right | \ll_{\epsilon} \frac{(1+k_2-k_1)^{1/2}}{(\log k_2)^{1/2-\epsilon} L(1, \sym ^2 f)^{1/2} L(1, \sym^2 g)^{1/2}} .
\end{equation*}
For $N_{\epsilon}$ large depending on $\epsilon$, we have
\begin{align*}
     \left | \inprod{u_j R_{k_1}^{k_2} F_{k_1}}{G_{k_2}} \right | \ll_{\epsilon, t_j} \begin{cases}   \displaystyle \frac{1}{(\log k_2)^{1/2-\epsilon} L(1, \sym ^2 f)^{1/2} L(1, \sym^2 g)^{1/2}} & \text{if } k_2-k_1 \leq N_{\epsilon}; \\ k_2^{-1+\epsilon} \quad & \text{if } k_2-k_1 \geq N_{\epsilon}. \end{cases}
\end{align*}
\end{lemma}

\begin{proof}

From Ichino's formula \cite{Ichino08}, we know that
\begin{equation}
\label{cusp triple product}
    \left |  \int_{\GH} u_{j,k_2-k_1}(z) F_{k_1}(z) \overline{G_{k_2}(z)} d \mu(z)\right |^2 = \frac{1}{8} \frac{\Lambda(1/2, u_j \times f \times g)}{\Lambda(1, \sym ^2 u_j) \Lambda(1, \sym ^2 f) \Lambda (1, \sym^2 g)} I_{\infty}^* ,
\end{equation}
where $I_{\infty}^*$ is a certain local integral. When $k_1=k_2$, Watson \cite{Watson} shows that $I_{\infty}^*=1$. For the general case, Woodbury \cite{Woo17} and Cheng \cite{Cheng21} calculated for the real local place and show that $I_{\infty}^*=2^{-k_2+k_1}$.

We have that
\begin{align*}
    \Lambda(s, f \times g \times u_j) = & \prod_{\pm}  \Gamma_{\R}\lr{s+\frac{k_1+k_2}{2} \pm it_j}\Gamma_{\R}\lr{s+\frac{k_1+k_2}{2}-1 \pm it_j}  \\
    & \times \Gamma_{\R}\lr{s+\frac{k_2-k_1}{2} \pm it_j}\Gamma_{\R}\lr{s+\frac{k_2-k_1}{2} +1 \pm ir} L(s, f\times g \times u_j) .
\end{align*}
Then it follows that
\begin{align}
\label{peter equation}
   \left | \inprod{u_{j, k_2-k_1} F_{k_1}}{G_{k_2}} \right |^2 & \ll_{t_j}   \frac{ \Gamma   \lr{\frac{k_1+k_2-1}{2} + i t_j}    \Gamma \lr{\frac{k_1+k_2-1}{2} - i t_j}L(1/2, f \times g \times u_j)}{ \Gamma(k_1) \Gamma(k_2) L(1, \sym^2 f) L(1, \sym^2 g)}   .
\end{align}
We use the weak subconvexity bound \cite{Sound}
    \begin{equation*}
        L \lr{\frac{1}{2}, f \times g \times u_j}\ll_{t_j, \epsilon} \frac{(k_1+k_2)(1+k_2-k_1)}{(\log k_2)^{1-\epsilon}}.
    \end{equation*}
Similarly to the previous proof, we use that for $\sigma \geq 1/2$, we have that $\Gamma(\sigma+1/2) \asymp \sqrt{\sigma} \Gamma(\sigma)$ and $|\Gamma(\sigma+it_j)| \leq \Gamma(\sigma)$. Also, as before, we know that $\Gamma \lr{\frac{k_1+k_2}{2}}^2 \leq \Gamma(k_1) \Gamma(k_2)$ and then we conclude the first part of the lemma.

For the second part, we first note that
\begin{align}
\label{atat}
\begin{split}
    \left | \inprod{u_j R_{k_1}^{k_2} F_{k_1}}{G_{k_2}} \right |^2 &= \beta(k_1,k_2)^2 \left | \inprod{(K_{k_2-k_1-2} \dots K_0 u_j) F_{k_1}}{G_{k_2}}\right |^2 \\
    &= \frac{\beta(k_1,k_2)^2}{\alpha(s_j, k_2-k_1)^2} |\inprod{u_{j, k_2-k_1} F_{k_1}}{G_{k_2}}|^2 .
    \end{split}
\end{align}

Now fix $N_{\epsilon}$ large enough such that $N_{\epsilon}> 1 / \epsilon$ and $\log n < n^{\epsilon}$, for $n \geq N_{\epsilon}$. We treat two separate cases, depending on whether $k_2-k_1$ is smaller or larger than $N_{\epsilon}$.

\begin{enumerate}
    \item If $0\leq k_2-k_1\leq N_{\epsilon}$. Then from definitions of \eqref{alpha} \eqref{beta}, we see that
    \begin{equation*}
        \frac{\beta(k_1,k_2)^2}{\alpha(s_j, k_2-k_1)^2} \ll_{\epsilon, t_j} 1
    \end{equation*}
    and the conclusion follows.

\item If $k_2-k_1 \geq N^{\epsilon}$. For notation simplicity, denote $\alpha=(k_2-k_1)/2$.  We also use the bounds $L(1, \sym^2 f) \gg (\log k_1)^{-1}$ and $L(1, \sym^2 g) \gg (\log k_2)^{-1}$. 
Now, from \eqref{alpha} and \eqref{beta}, we see that
\begin{align*}
    \left |\frac{\beta(k_1,k_2)^2}{\alpha(1/2+i t_j, k_2-k_1)^2} \right | &\ll_{t_j} \frac{\Gamma(k_1)  \Gamma \lr{\alpha +\frac{1}{2}}^2 }{\Gamma{\lr{\frac{k_1+k_2}{2}}} \Gamma \lr{\alpha+1}} = \frac{\Gamma(k_1) \Gamma (\alpha+1)}{\Gamma{\lr{\frac{k_1+k_2}{2}}}} \frac{\Gamma \lr{\alpha+\frac{1}{2}}^2}{\Gamma(\alpha+1)^2} \\ &\ll {\frac{k_1+k_2}{2}-1 \choose \alpha}^{-1} \frac{1}{\alpha} \ll k_2^{-1} \alpha^{-1}
\end{align*}
Now the conclusion follows from \eqref{atat}.
\end{enumerate}
\end{proof}

\section{Bounds for Fourier coefficients}
\label{Fourier bounds}

In order to evaluate Fourier coefficients of automorphic forms of weight $k$, it is useful to define
\begin{equation}
    F(k,t,y):= \frac{W_{k,it}(u)}{\Gamma\lr{\frac{1}{2}+k+it}}+  \frac{W_{-k,it}(u)}{\Gamma\lr{\frac{1}{2}-k+it}}. 
\end{equation}
In \cite{Jakobson_cusp}, Jakobson evaluated this expression as
\begin{equation*}
    F(k,t,y)=2 (-1)^k \sum_{l=0}^k\frac{(-k)_l (k)_l y^l}{(1/2)_l 4^l l!} \frac{W_{0, l+it}(y)}{\Gamma \lr{\frac{1}{2} +l + it}} ,
\end{equation*}
where the Pochhammer symbol $(x)_l$ is defined by
\begin{equation*}
    (x)_l:=x (x+1)\dots (x+l-1) ; \quad (x)_0 = 1.
\end{equation*}

We use the fact that $W_{0,\nu}(y)=\sqrt{(y/\pi)} K_{\nu}(y/2)$. We apply the integral representation  of the $K$-Bessel function \cite[p. 205]{Iw}
\begin{equation*}
    K_{\nu}(y) = \pi^{-1/2} \Gamma \lr{\nu + \frac{1}{2}} \lr{\frac{y}{2}}^{-\nu} \int_{0}^{\infty} (u^2+1)^{-\nu -1/2} \cos (u y) du,
\end{equation*}
which holds for $y>0$ and $\Re(\nu)>-1/2$. From this we obtain 
\begin{align*}
    \frac{y^l W_{0, l+ it} (y)}{\Gamma \lr{\frac{1}{2} +l + it}} &\ll  y^{1/2} \left | \int_{0}^{\infty} (u^2+1)^{-l -1/2 -it} \cos (u y) du \right | \\ &\ll y^{1/2} \lr{\frac{1+l+|t|}{y}}^{A} \lr{1+ \frac{1+|t|}{y}}^{\epsilon},
\end{align*}
for any $\epsilon>0$ and any integer $A \geq 0$. 

Next we note that 
 \begin{align*}
     \left | \frac{(-k)_l (k)_l y^l}{(1/2)_l 4^l l!} \right | = \frac{k}{k+l} {k+l \choose l},
 \end{align*}
hence using the identity
\begin{equation*}
    \sum_{l=0}^m {k+l \choose l} = {k+m+1 \choose m},
\end{equation*} we see that
\begin{equation}
    \label{F(k,t,y) bound}
    F(k,t,y) \ll {4^k} k^{A} \sqrt{y}  \lr{\frac{1+|t|}{y}}^{A} \lr{1+ \frac{1+|t|}{y}}^{\epsilon} .
\end{equation}
 Also, from \cite[B. 36]{Iw2}, we have the asymptotic for large $y$ 
\begin{equation*}
    K_\nu(y)=\lr{\frac{\pi }{2y}}^{1/2} e^{-y} \lr{1 + O \lr{\frac{1+|\nu|^2}{y}}} .
\end{equation*}

Now we are ready to give bounds for the Fourier coefficients of incomplete Eisenstein series.
\begin{lemma}
\label{bound eisenstein}
Let $E_k( z | \psi)$ an incomplete Eisenstein series with Fourier expansion
\begin{equation*}
     E_{k}(z | \psi) = \sum_{n \in \Z} a_n(y) e(nx).
\end{equation*}
Then
\begin{equation*}
    a_0(y) = \delta_{k=0} \frac{3}{\pi} \Psi(-1) + O \lr{\sqrt{y}},
\end{equation*}
and for $n \neq 0$, we have
\begin{equation*}
    a_n(y)+ a_{-n}(y) \ll  {2^k} k^{A} \sqrt{y}\tau(|n|) \lr{\frac{1}{|n| y}}^A \lr{1 + \frac{1}{|n| y}}^{\epsilon},
\end{equation*}
for any $\epsilon>0$ and any integer $A \geq 0$. 
\end{lemma}

\begin{proof}
Using \eqref{blabla} and \eqref{eisenstein expansion}, we note that
\begin{equation*}
    a_0(y) = \frac{1}{2 \pi i} \int_{(\sigma)} \Psi(-s) \lr{y^{s} + \frac{(-1)^{k/2} \Gamma \lr{s}^2}{\Gamma \lr{s-\frac{k}{2}}\Gamma \lr{s+\frac{k}{2}}} \phi(s) y^{1- s}} ds,
\end{equation*}
for some $\sigma>1$. We want to move the line of integration to $\Re(s)=1/2$ and we notice we that encounter a pole at $s=1$ if and only if $k=0$. Using the duplication formula $\Gamma(z) \Gamma(1-z) = \frac{\pi}{\sin \pi z}$ and that $|\Gamma(1/2 + it)|^2 = \frac{\pi}{\cosh{\pi t}}$, we observe that
$$ \left | \frac{\Gamma \lr{\frac{1}{2} + it}^2}{\Gamma \lr{\frac{1}{2} + it +\frac{k}{2} } \Gamma \lr{\frac{1}{2} + it -\frac{k}{2} }} \right | = \left|  \frac{ \displaystyle \frac{\pi}{\cosh \pi t}}{ \displaystyle \frac{\pi}{\sin   \pi \lr{\frac{1}{2} +it + \frac{k}{2} }}}\right | \sim 1  \quad \text{as } |t| \to  \infty. $$
Hence, by \eqref{decay}, we have that
\begin{equation}
    \label{O coeff incomplete Eisen}
    a_0(y) = \delta_{k=0} \frac{3}{\pi} \Psi(-1) + O \lr{\sqrt{y}}.
\end{equation}
Note that, by unfolding, we see that
\begin{equation*}
    \inprod{E_0(z | \psi)}{1}= \int_{-1/2}^{1/2} \int_0^{\infty} \psi(y) \frac{dx dy}{y^2} = \Psi(-1).
\end{equation*}
Similarly, for $n \neq 0$, we have that
\begin{equation*}
    a_n(y)= \frac{1}{2 \pi i} \int_{- \infty}^{\infty}  \Psi \lr{ -\frac{1}{2} - it} \frac{(-1)^{k/2} \Gamma \lr{\frac{1}{2} + it}}{2 \Gamma \lr{\frac{1}{2} +\frac{k}{2} + it} \xi(1 + 2 i t)} |l|^{-\frac{1}{2}} \lr{\sum_{ab=|n|} \lr{\frac{a}{b}}^{it}} W_{\frac{k}{2}, it}(4 \pi |n| y) dt .
\end{equation*}
We easily see that
$$a_n(y) + a_{-n}(y) \ll \tau(|n|) |n|^{-1/2}  \int_{- \infty}^{\infty}  \Psi \lr{ -\frac{1}{2} - it} \frac{ \Gamma \lr{\frac{1}{2} + it}}{ \xi(1 + 2 i t)} F\lr{\frac{k}{2}, t, 4 \pi |n| y} . $$
The conclusion follows from \eqref{F(k,t,y) bound}.
\end{proof}

Next we turn our attention to the Fourier coefficients of Maa{\ss} cusp forms.

\begin{lemma}
\label{bound cusp}
Let $u_{j,k}$ be a Maa{\ss} cusp form as defined in the previous section with eigenvalue $1/4 + t_j^2$. If its  Fourier expansion is given by
$$ u_{j,k}(z) = \sum_{n \in \Z} a_n(y) e(nx), $$ then $a_0(y) = 0$ and for $n\neq 0$, we have that
$$ a_n(y) + a_{-n}(y) \ll 2^k k^{A} \sqrt{y} |c_j(|n|)| \lr{\frac{1 + |t_j|}{|n| y}}^A \lr{1 + \frac{1 + |t_j|}{|n| y}}^{\epsilon} .$$
\end{lemma}
\begin{proof}
From  \eqref{fourier maass cusp k}, we see that for $n \neq 0$, we have that
\begin{equation*}
    a_n(y) + a_n(-y) = \Gamma (1/2 + i t_j) c_j(|n|) |n|^{-1/2} F(k/2, t_j, 4 \pi |n| y). 
\end{equation*}
Now the conclusion simply follows from \eqref{F(k,t,y) bound}.
\end{proof}

Next we develop a formula for Whittaker functions of the form $W_{\alpha+k, \alpha-\frac{1}{2}}(y)$, which is useful for expressing the Fourier coefficients of $R_{k_1}^{k_2}F_{k_1}$.
\begin{lemma}
\label{whittaker hol cusp}
Let $\alpha>0$  and $k \geq 0$ an integer. Then
\begin{equation*}
    W_{\alpha+k, \alpha-\frac{1}{2}}(y) = e^{-\frac{y}{2}} y^{\alpha} \sum_{l=0}^k y^{k-l} (-1)^l {k \choose l} \frac{\Gamma (2 \alpha +k)}{\Gamma (2 \alpha + k -l)}.
\end{equation*}
In particular, this implies that for $y>0$ and $\alpha \geq 1$, we have
\begin{equation*}
     W_{\alpha+k, \alpha-\frac{1}{2}}(y) \ll 2^k e^{-\frac{y}{2}} y^{\alpha} ((2\alpha+k)^k + y^k).
\end{equation*}
\end{lemma}

\begin{proof}
We proceed by induction on $k$. From \cite[(4.21)]{DFI02}, we see that
\begin{equation*}
    W_{\alpha, \alpha-\frac{1}{2}}(y) = y^{\alpha} e^{-y/2} .
\end{equation*}
We use the recursion formula \cite[(9.234)]{GR}
\begin{equation*}
    W_{\lambda+1, \mu}(y) = \lr{\frac{1}{2}y - \lambda}W_{\lambda, \mu}(y) - y W'_{\lambda, \mu}(y).
\end{equation*}
We we see that $W_{\alpha+k, \alpha-\frac{1}{2}}(y)$ is of the form
$$W_{\alpha+k, \alpha-\frac{1}{2}}(y) = e^{-\frac{y}{2}} y^{\alpha} \sum_{l=0}^k y^{k-l} P_{k,l}(\alpha)$$
where $P_{k,l}(X)$ polynomial of degree $l$. The recursion formula gives us that, for $1 \leq l \leq k$, we have 
\begin{equation*}
    P_{k+1,l} (\alpha) = P_{k,l} (\alpha) - (2 \alpha + 2k-l+1) P_{k, l-1} (\alpha).
\end{equation*}
Moreover, $P_{k,0} (\alpha)=1$ and $P_{k,k}(\alpha) = (-1)^k ( 2 \alpha)_k$, for all $k$. If we write $Q_{k,l}(X)=P_{k,l}\lr{\frac{X}{2}}$, one can check by inducion on $k$ that
\begin {equation*}
    Q_{k,l}(X) = (-1)^l {k \choose l}(X+k-1)(X+k-2)\cdots (X+k-l).
\end{equation*}
The conclusion follows.
\end{proof}

\section{Shifted convolution sums}
\label{holowinsky}

Let $\phi \in \mathcal{L}_{k_2-k_1}(X)$ with Fourier expansion
\begin{equation*}
    \phi(z)=a_0(y) + \sum_{l \neq 0} a_l(y) e(lx).
\end{equation*}
We want to evaluate $\inprod{\phi F_{k_1}}{G_{k_2}}$ by applying Holowinsky's approach \cite{Hol} by relating the inner product to shifted convolution sums. In this section we prove the following theorem.

\begin{theorem}
\label{shifted main}
Define
\begin{equation}
        \label{M(f,g)}
        M_{k_1,k_2}(f,g):=\frac{1}{(\log k_2)^2 L(1, \sym^2 f)^{1/2} L(1, \sym^2 g)^{1/2}} \prod_{p \leq k_2} \lr{1+\frac{|\lambda_f(p)|}{p}} \lr{1+\frac{|\lambda_g(p)|}{p}}.
    \end{equation}
Fix $\epsilon>0$. Then there exists a constant $N_{\epsilon}$ such that the following hold.
    \begin{enumerate}[label=\roman*]
        \item Let $u_{j,k_2-k_1}$ be a Hecke--Maa{\ss} form as above with eigenvalue $1/4 + t_j^2$. Then
        \begin{equation*}
            \inprod{u_{j, k_2-k_1}F_{k_1}}{G_{k_2}}\ll_{t_j, \epsilon} 2^{k_2-k_1} (1+k_2-k_1)^{N_{\epsilon}}M_{k_1,k_2}(f)^{1/2} (\log k_2)^{\epsilon}.
        \end{equation*}
        
        \item For an incomplete Eisenstein series $E_{k_2-k_1}(z |\psi)$, we have that $\inprod{E_{k_2-k_1}(\cdot |\psi) F_{k_1}}{G_{k_2}} - \delta_{f=g}\frac{3}{\pi} \inprod{E_0(\cdot | \psi)}{1}$ is bounded by
        \begin{equation*}
            O_{\psi, \epsilon} \lr{ 2^{k_2-k_1}(1+k_2-k_1)^{N_{\epsilon}} M_{k_1,k_2}(f)^{1/2} (\log k_2)^{\epsilon} (1+R_{k_1,k_2}(f,g))},
        \end{equation*}
        where
        \begin{equation*}
            R_{k_1,k_2}(f,g)=\frac{1}{k_2^{1/2} L(1, \sym^2 f)^{1/2} L(1, \sym^2 g)^{1/2}} \int_{- \infty} ^{+ \infty} \frac{\left |  L \lr{f \times g, \frac{1}{2}+it} \right|}{(|t|+1)^5} dt .
        \end{equation*}
    \end{enumerate}
\end{theorem}

Fix $\psi$ smooth and compactly supported on $\R^+$ and $\Psi(s)$ its Mellin transform.
For $Y \geq 1$, we define
\begin{equation}
    \label{Iphi2}
    I_{\phi}(Y) := \frac{1}{2 \pi i} \int_{(\sigma)} \Psi(-s) Y^s \int_{X}  E(z,s) \phi(z) F_{k_1}(z) \overline{G_{k_2}(z)} d \mu ds
\end{equation}
for $\sigma >1$.

\begin{lemma}
For $\phi$ a fixed a Hecke--Maa{\ss} cusp form or incomplete Eisenstein series, we have
\begin{equation}
    \label{inprodsum}
    \inprod{\phi F_{k_1}}{G_{k_2}} = c_Y^{-1} I_{\phi}(Y) + O_{ \psi} \lr{Y^{-1/2}} \ ,
\end{equation}
where
\begin{equation}
    \label{cY}
    c_Y := \frac{3}{\pi} \Psi(-1) Y \ .
\end{equation}
\end{lemma}

\begin{proof}

We move the contour of integration in \eqref{Iphi2} to the line $\Re(s)=1/2$. There is a pole at $s=1$ coming from the Eisenstein series, with residue
\begin{equation*}
    \Psi(-1) Y \lr{\Res_{s=1}E(z,s)} \inprod{\phi F_{k_1}}{G_{k_2}} = c_Y  \inprod{\phi F_{k_1}}{G_{k_2}} \ .
\end{equation*}
Therefore we obtain 
\begin{equation}
    \label{there}
    I_{\phi}(Y) = c_Y \inprod{\phi F_{k_1}}{G_{k_2}} + \int_{X}p(z) \phi (z) F_{k_1}(z) \overline{G_{k_2}(z)} d \mu \ ,
\end{equation}
where
\begin{equation*}
    p(z):= \int_{(1/2)} \Psi(-s) Y^s E(z,s) ds \ .
\end{equation*}
On the line $\Re(s)=1/2$, from \cite[Lemma 2.1]{Hol}, we have $$E(z,s) \ll_{\epsilon} \sqrt{y} + |s|^2 y^{-3/2}(1+|s|/ y)^{\epsilon}.$$
Using the fast decay of $\Psi(s)$, we obtain $p(z) \ll \sqrt{yY}$ if $y \geq 1/2$. Going back to \eqref{there}, if we assume $\sqrt{y} |\phi(z)|$ is bounded on $X$, we conclude that 
\begin{equation*}
     \int_{X}p(z) \phi (z) F_{k_1}(z) \overline{G_{k_2}(z)} d \mu \ll_{\phi, \psi} \sqrt{Y} \ .
\end{equation*}
The assumption that $\sqrt{y} |\phi(z)|$ is bounded on $X$ is true for cusp forms and incomplete Eisenstein series. 
\end{proof}
We observe that
\begin{equation}
    \label{Iphi}
    I_{\phi}(Y)= \frac{1}{2 \pi i} \int_0^{\infty} \psi(Yy) y^{-2} \lr{\int_{-1/2}^{1/2}\phi(z) F_{k_1}(z) \overline{G_{k_2}(z)}dx} dy \ .
\end{equation}
This follows from using a standard unfolding argument and then applying the inverse Mellin transform. 

\begin{proposition}
Let $Y>1$. For any $\epsilon>0$, there exists a constant $N_{\epsilon}$ such that, for $\phi$ a Hecke--Maa{\ss} cusp form or incomplete Eisenstein series, we have
\begin{align*}
    \inprod{\phi F_{k_1}}{G_{k_2}}= &c_{Y}^{-1} \int_{0}^\infty \psi(Yy) y^{-2} \lr{\int_{-1/2}^{1/2} \phi^*(z) F_{k_1}(z) \overline{G_{k_2}(z)} dx} dy  \\ &+ O \lr{ 2^{k_2-k_1}{(k_2-k_1 +1)^{N_{\epsilon}}}Y^{-1/2}} , 
\end{align*}
where
\begin{equation*}
    \phi^*(z):= \sum_{|l|<Y^{1+ \epsilon}} a_l(y) e(lx) \ .
\end{equation*}
\end{proposition}

\begin{proof}
We evaluate the contribution to $I_{\phi}(Y)$ coming from large Fourier coefficients $a_l(y)$ of $\phi$. Assume $\phi$ is an incomplete Eisenstein series of weight $k_2-k_1$. We make use of Lemma \ref{bound eisenstein}. The contribution coming from Fourier coefficients larger than $Y^{1+\epsilon}$ is bounded by
\begin{align*}
    &\sum_{|l|\geq Y^{1+ \epsilon}} \int_{0}^\infty  \int_{-1/2}^{1/2}  \psi(Yy) y^{-2}  a_l(y)  |F_{k_1}(z)|  |G_{k_2}(z)| dx dy \\
    \ll &  2^{k_2-k_1}{(k_2-k_1+1)^{A}}\lr{  \int_{0}^\infty  \int_{-1/2}^{1/2}  \psi(Yy) y^{-2}    |F_{k_1}(z)|  |G_{k_2}(z)| dx dy} Y^{A -1/2 + \epsilon} \sum_{l > Y^{1+ \epsilon}} \frac{\tau(l)}{l^A} \\
    \ll &  {2^{k_2-k_1}}{(k_2-k_1+1)^{A}} Y^{A+1/2+ \epsilon} Y^{(1+\epsilon)(1-A)} \ll  {2^{k_2-k_1}}{(k_2-k_1+1)^{A}} Y^{-1/2} ,
\end{align*}
if we choose $A$ large enough with respect to $\epsilon$. We note that the double integral is bounded by $O(Y)$, since
\begin{equation*}
    \int_{\GH} |F_{k_1}(z)| |G_{k_2}(z)| d \mu (z) \leq \| F_{k_1} \| + \| G_{k_2} \| = 2,
\end{equation*}
and we know that $y \asymp 1/Y$, $-1/2 \leq x \leq 1/2$, and by \cite[Lemma 2.10]{Iw} we know there are $O(Y)$ copies of the fundamental domain in this region. The proof for Maa{\ss} forms follows similarly.
\end{proof}

For an integer $l$, we define
\begin{equation}
    \label{S_l}
    S_l(Y):=  \int_{0}^\infty \psi(Yy) y^{-2} \lr{\int_{-1/2}^{1/2} a_l(y) e(lx) F_{k_1}(z) \overline{G_{k_2}(z)} dx} dy \ .
\end{equation}
Hence
\begin{equation}
    c_Y \inprod{\phi F_{k_1}}{G_{k_2}} = S_0(Y) + \sum_{0 < |l|<Y^{1+\epsilon}} S_l(Y) + O\lr{ 2^{k_2-k_1} {(k_2-k_1 +1)^{N_{\epsilon}}} Y^{1/2}} \ .
\end{equation}
We note that $S_0(Y) \equiv 0$ when $\phi$ is a cusp form.

\begin{lemma}
\label{S_0}
Let $Y \geq 1$ and $\phi=E_{k_2-k_1}(z | h)$ an incomplete Eisenstein series of weight $k_2-k_1$. Then
\begin{align*}
   c_Y^{-1} S_0(Y) = &\delta_{k_1=k_2} \frac{3}{\pi} \inprod{\phi}{1}  + O(Y^{-1/2}) \\ & + O\lr{\left (Y k_2)^{-1/2}(|L(\sym^2 f, 1) L(\sym^2g, 1)| \right )^{-1/2} \int_{-\infty}^{\infty} \frac{\left |  L \lr{f \times g, \frac{1}{2}+it} \right|}{(|t|+1)^5} dt} .
\end{align*}
\end{lemma}
\begin{proof}
From the definition of $S_0(Y)$ and \eqref{O coeff incomplete Eisen}, we obtain
\begin{equation*}
    S_0(Y)= \lr{\delta_{k_1=k_2} \frac{3}{\pi} \inprod{\phi}{1} + O \lr{Y^{-1/2}}}\int_0^{\infty} \psi(Yy)y^{\frac{k_1+k_2}{2}-2} \lr{\int_{-1/2}^{1/2} f(z) \overline{g(z)} dx} dy \ .
\end{equation*}
Expanding the product $f(z) \overline{g(z)}$ as a Fourier sum and computing the inner integral above, we obtain
\begin{equation*}
    S_0(Y) = \lr{ \delta_{k_1=k_2} \frac{3}{\pi} \inprod{\phi}{1}+ O \lr{Y^{-1/2}}} \sum_{n \geq 1} a_f(n) \overline{a_g(n)} \int_0^{\infty} \psi(Yy) y^{\frac{k_1+k_2}{2}-2} e^{- 4 \pi n y} dy \ .
\end{equation*}
We evaluate the integral in $y$ using the inverse Mellin transform.
\begin{align*}
    \int_0^{\infty} \psi(Yy) y^{\frac{k_1+k_2}{2}-2} e^{- 4 \pi n y} dy &= \int_0^{\infty}  \lr{\frac{1}{2 \pi i}\int_{(\sigma)}(Yy)^s \Psi(-s) ds}  y^{\frac{k_1+k_2}{2}-2} e^{- 4 \pi n y} dy \\
    &= \frac{1}{2 \pi i}\int_{(\sigma)} Y^s \Psi(-s) (4 \pi n)^{-s-\frac{k_1+k_2}{2}+1} \Gamma\lr{s+\frac{k_1+k_2}{2}-1} ds,
\end{align*}
Hence 
\begin{align*}
    &\sum_{n \geq 1 }a_f(n) \overline{a_g(n)} \int_0^{\infty} \psi(Yy)   y^{\frac{k_1+k_2}{2}-2} e^{- 4 \pi n y} dy = \\ =& \frac{1}{2 \pi i} a_f(1) \overline{a_g(1)} (4 \pi)^{1-\frac{k_1+k_2}{2}}
     \int_{(\sigma)} \lr{\frac{Y}{4 \pi}}^s \Psi(-s) \frac{L(f \times g, s)}{\zeta(2s)} \Gamma\lr{s+\frac{k_1+k_2}{2}-1} ds \ .
\end{align*}
We move the contour of integration to the line $\Re(s)=1/2$. We note that we pick up a pole at $s=1$ if and only if $f=g$. In this case, we use \eqref{af1} to compute the residue. Therefore, we obtain
\begin{align*}
    \sum_{n \geq 1 }a_f(n) \overline{a_g(n)} \int_0^{\infty} \psi(Yy) y^{\frac{k_1+k_2}{2}-2} e^{- 4 \pi n y} dy = \delta_{f=g} \frac{3}{\pi} \Psi(-1) Y +  E(Y) ,
\end{align*}
where
\begin{align*}
    E(Y) =& \frac{1}{2 \pi i} a_f(1) \overline{a_g(1)} (4 \pi)^{1-\frac{k_1+k_2}{2}} \\ & \times 
     \int_{-\infty}^{\infty} \lr{\frac{Y}{4 \pi}}^{1/2+it} \Psi \lr{-\frac{1}{2} -it} \frac{L\left (f \times g, \frac{1}{2}+it \right )}{\zeta(1+2 i t)} \Gamma\lr{\frac{k_1+k_2}{2}-\frac{1}{2} + it} dt .
\end{align*}
From \cite[p.51]{tit}, we know that $\zeta(1+it) \gg (\log t)^{-7}$. 
Hence, using the rapid decay of $\Psi(-s)$ guaranteed by \eqref{decay} and expanding $a_f(1) \overline{a_g(1)}$ as in \eqref{af1}, we obtain
\begin{align*}
    E(Y) &\ll Y^{1/2} \frac{{k_1+k_2-2 \choose k_1-1}^{1/2}}{(k_1+k_2)^{1/2} {k_1+k_2-2 \choose {\frac{k_1+k_2}{2}-1}}^{1/2}} |L(\sym^2 f, 1) L(\sym^2 g, 1)|^{-1/2} \int_{-\infty}^{\infty} \frac{\left |  L \lr{f \times g, \frac{1}{2}+it} \right|}{(|t|+1)^{10}} dt \\
    &\ll Y^{1/2} k_2^{-1/2}  |L(\sym^2 f, 1) L(\sym^2 g, 1)|^{-1/2} \int_{-\infty}^{\infty} \frac{\left |  L \lr{f \times g, \frac{1}{2}+it} \right|}{(|t|+1)^{10}} dt .
\end{align*}
\end{proof}

\begin{lemma}
\label{S_l(Y)}
Let $\phi$ be a fixed automorphic form. Then for $l \neq 0$, we have
\begin{equation*}
    c_Y^{-1} S_l(Y) \ll  \left | \frac{a_l(Y^{-1})}{L(\sym^2 f, 1)^{1/2} L(\sym^2 g , 1)^{1/2}} \right | \lr{\frac{1}{ Y k_2}\sum_{n \asymp Y k_2 } |\lambda_f(n) \lambda_g(n+l)| + Y^{\epsilon}(k_1+k_2)^{-1+\epsilon}} .
\end{equation*}
\end{lemma}

\begin{proof}
Expanding the Fourier sum in the definition \eqref{S_l}, we obtain
\begin{equation*}
    S_l(Y)=\sum_{n \geq 1} a_f(n) \overline{a_g(n+l)} \int_{0}^\infty \psi(Yy) a_l(y) y^{\frac{k_1+k_2}{2}-2}e^{- 2 \pi (2n+l) dy} \ .
\end{equation*}
We note that the inner integral is only supported for $y \asymp 1/Y$. Hence
\begin{equation*}
     S_l(Y) \ll |a_l(Y^{-1})| \sum_{n \geq 1} a_f(n) \overline{a_g(n+l)} \int_{0}^\infty \psi(Yy) y^{\frac{k_1+k_2}{2}-2}e^{- 2 \pi (2n+l)} dy \ .
\end{equation*}
Similarly as in the proof of Lemma \ref{S_0}, using the inverse Mellin transform and evaluating the inner integral, we obtain
\begin{align*}
    S_l(Y) \ll |a_l(Y^{-1})|  \sum_{n \geq 1} a_f(n) \overline{a_g(n+l)} \frac{1}{2 \pi i} \int_{(\sigma)} Y^s \Psi(-s) (2 \pi (2n+l))^{1-s-\frac{k_1+k_2}{2}} \Gamma \lr{s+ \frac{k_1+k_2}{2}-1} ds .
    \end{align*}
From \eqref{af1}, we see that
\begin{equation*}
    S_l(Y) \ll \left | \frac{a_l(Y^{-1})}{L(\sym^2 f, 1)^{1/2} L(\sym^2 g , 1)^{1/2}} \right | \sum_{n \geq 1} |\lambda_f(n) \lambda_g(n+l)| A_{n,l}(Y) \ ,
\end{equation*}
where
\begin{equation*}
    A_{n,l}(Y):=\lr{\frac{n^{\frac{k_1-1}{2}}(n+l)^{\frac{k_2-1}{2}}}{\lr{n+\frac{l}{2}}^{\frac{k_1+k_2}{2}-1}}} \frac{1}{2 \pi i} \int_{(\sigma)} \Psi(-s) \lr{\frac{Y}{2 \pi (2n+l)}}^s \frac{\Gamma(s+ \frac{k_1+k_2}{2}-1)}{\Gamma(k_1)^{1/2} \Gamma(k_2)^{1/2}} ds .
\end{equation*}
If we interchange $f$ and $g$, which we can without losing the generality, then the first term will be bounded above by 1. From Stirling's relations, any vertical strip $0<a \leq \Re (s) \leq b$ and $k>1$, we have
\begin{equation}
    \label{gamma}
    \frac{\Gamma(s+\alpha)}{\Gamma(\alpha)}=\alpha^s \lr{1+O_{a,b}\lr{(|s|+1)^2 \alpha^{-1}}} ,
\end{equation}
see \cite[(19)]{Hol}. Choosing the line of integration $\Re(s)=\sigma=1+\epsilon$, we obtain
\begin{align*}
    A_{n,l}(Y) &\ll  \frac{1}{2 \pi i} \int_{(\sigma)} \Psi(-s) \lr{\frac{Y}{2 \pi (2n+l)}}^s \frac{\Gamma(s+ \frac{k_1+k_2}{2}-1)}{\Gamma \lr{\frac{k_1+k_2}{2}-1}}\frac{\Gamma \lr{\frac{k_1+k_2}{2}-1}}{\Gamma(k_1)^{1/2} \Gamma(k_2)^{1/2}} ds \\
    & \ll \frac{{k_1+k_2-2 \choose k_1-1}^{1/2}}{(k_1+k_2) {k_1+k_2-2 \choose \frac{k_1+k_2}{2} -1}^{1/2}}  \lr{\psi \lr{\frac{Y \lr{\frac{k_1+k_2}{2}-1}}{2 \pi (2n+l)}} + (k_1+k_2)^{\epsilon} \lr{\frac{Y}{2n+l}}^{1+\epsilon}}.
\end{align*}
 Therefore we get
\begin{equation*}
    S_l(Y) \ll \left | \frac{a_l(Y^{-1})}{L(\sym^2 f, 1)^{1/2} L(\sym^2 g , 1)^{1/2}} \right | \lr{\frac{1}{ k_2}\sum_{n \asymp Y k_2 } |\lambda_f(n) \lambda_g(n+l)| + Y^{1+\epsilon}(k_1+k_2)^{-1+\epsilon}} .
\end{equation*}
\end{proof}



We recall \cite[Theorem 1.2]{Hol}.

\begin{theorem}
\label{sieve}
Let $\lambda_1(n)$ and $\lambda_2(n)$ be multiplicative functions such that $|\lambda_i(n)|\leq \tau(n)$. Then for any $0<\delta<1$ and any fixed integer $0<|l| \leq x$, we have
\begin{equation*}
    \sum_{n \leq x}|\lambda_1(n) \lambda_2(n+l)| \ll x (\log x)^{-2+ \delta} \tau(|l|) \prod_{p \leq z} \lr{1+\frac{|\lambda_1(p)|}{p}}\lr{1+\frac{|\lambda_2(p)|}{p}}
\end{equation*}
where $z=\exp \lr{\frac{\log x}{\delta \log \log x}}$.
\end{theorem}

We apply Theorem \ref{sieve} with $\lambda_1=\lambda_f$ and $\lambda_2=\lambda_g$. The Ramanujan conjecture for holomorphic cusp forms ensures that the conditions in the statement of the theorem are satisfied. There exists a constant $C_{\psi}$ such that, for all $\epsilon>0$
\begin{align*}
    &\sum_{n} |\lambda_f(n) \lambda_g(n+l)| \psi \lr{\frac{Y \lr{\frac{k_1+k_2}{2}-1}}{2 \pi (2n+l)}}    \ll  \sum_{n \leq C_{\psi}Y^(k_1+k_2)} \lambda_f(n) \lambda_g(n+l) \\
    \ll &  \tau(|l|)Y^{1 + \epsilon} (k_1+k_2) (\log (k_1+k_2))^{-2+\epsilon} \prod_{p \leq (k_1 +k_2)^{\epsilon}} \lr{1+\frac{|\lambda_1(p)|}{p}}\lr{1+\frac{|\lambda_2(p)|}{p}} .
\end{align*}

    {\it Case 1}: $\phi$ is an incomplete Eisenstein series. Using Lemma \ref{bound eisenstein}, we have that
    \begin{align*}
        S_l(Y) + S_{-l}(Y) \ll 2^{k_2-k_1} S(f,g) ^{-1/2}Y^{1/2+\epsilon}  \tau(l)^2   (\log k_2)^{-2+\epsilon} \prod_{p \leq k_2} \lr{1+\frac{|\lambda_1(p)|}{p}} \lr{1+\frac{|\lambda_2(p)|}{p}} .
    \end{align*}
 We use the trivial bound 
    $$\sum_{1 \leq l <Y^{1+ \epsilon}} \tau(l)^2 \ll Y^{1+\epsilon}$$ to see that
    \begin{equation*}
       C_Y^{-1} \sum_{0 < |l| < Y^{1+\epsilon}} S_l(Y) \ll 2^{k_2-k_1} Y^{1/2+\epsilon} M_{k_1,k_2}(f,g) .
    \end{equation*}

     {\it Case 2}: $\phi=u_{j, k_2-k_1}$ is a Hecke--Maa{\ss} cusp form. It is very similar to the above case, where we employ Lemma \ref{bound cusp} instead. While we sum $S_l(Y)$, we need to bound
    \begin{align*}
        \sum_{0<l<Y^{1+\epsilon}} \tau(l) c_j(l) \ll \lr{ \sum_{0<l<Y^{1+\epsilon}} \tau(l)^2}^{1/2} \lr{ \sum_{0<l<Y^{1+\epsilon}} c_j(l)^2}^{1/2} \ll Y^{1+\epsilon},
    \end{align*}
    where the bound for the second sum over the Hecke eigenvalues follows from \cite[p. 55]{Iw}.

To finish the proof of Theorem \ref{shifted main}, we simply choose $Y=M_{k_1,k_2}(f,g)^{-1}$. If $M_{k_1,k_2}(f,g)>1$, we take $Y=1$.  

\section{Proofs of Theorem \ref{main} and Theorem \ref{fixed difference}}
\label{proofs}

\begin{lemma}
If $k_1 \leq k_2$ and $\log k_1 \geq C \log k_2$, for some absolute constant $C$, then
\begin{equation*}
    M_{k_1,k_2}(f,g) \ll_{\epsilon} (\log k_2)^{1/6+ \epsilon} L(1, \sym ^2 f)^{\frac{1}{4}}   L(1, \sym ^2 g)^{\frac{1}{4}} .
\end{equation*}
\end{lemma}

\begin{proof}
The key input is \cite[Lemma 2]{HolSound} which states that
\begin{equation}
\label{lemma 2 hol sound}
    L(1, \sym^2 f) \gg (\log \log k_1)^{-3} \exp \lr{\sum_{p \leq k_1} \frac{\lambda_f(p^2)}{p}}, 
\end{equation}
and a similar statement holds for $L(1, \sym^2 g)$. As in \cite[Lemma 3]{HolSound}, we use the inequality $|x| \leq \frac{1}{3}+ \frac{3}{4} x^2$ and the Hecke relations $\lambda_f(p^2)  = \lambda_f(p)^2-1$ to see that
\begin{align*}
    \sum_{p \leq k_1} \frac{|\lambda_f(p)|}{p} &\leq \frac{1}{3} \sum_{p \leq k_1} \frac{1}{p} + \frac{3}{4} \sum_{p \leq k_1} \frac{\lambda_f(p)^2}{p} \\ &= \frac{13}{12} \sum_{p \leq k_1}\frac{1}{p} + \frac{3}{4} \sum_{p \leq k_1} \frac{\lambda_f(p^2)}{p} \\ &\leq \frac{13}{12} \log \log k_1 + \frac{3}{4} \sum_{p \leq k_1} \frac{\lambda_f(p^2)}{p} + O(1) .
\end{align*}
 Now the conclusion follows from \eqref{lemma 2 hol sound} and the fact that $\log k_1 \asymp \log k_2$.
\end{proof}

\subsection{Proof of Theorem \ref{fixed difference}}
From the analysis in Section \ref{theory weight k}, it suffices to bound $\inprod{u_{j,k_2-k_1}F_{k_1}}{G_{k_2}}$ and $\inprod{E_{k_2-k_1}(z | \psi) F_{k_1}}{G_{k_2}}$. We have who cases, depending on the size of $L(1, \sym^2 f) L(1, \sym^2 g)$.

{\it Case (i)}: Suppose $L(1, \sym^2 f) L(1, \sym^2 g) \geq (\log k_2)^{-5/6}$. Then by Lemma \ref{inner product cusp}, we have that
\begin{equation*}
    \left | \inprod{u_{j,k_2-k_1} F_{k_1}}{G_{k_2}} \right | \ll_{\epsilon}  \lr{1+ \frac{k_2-k_1}{2}}^{1/2} {(\log k_2)^{-1/12+ \epsilon}}.
\end{equation*}
For the Eisenstein case, from \eqref{blabla} we know that
\begin{equation*}
    E_{k_2-k_1}(z | \psi) = \delta_{k_1=k_2} \frac{3}{\pi}\Psi(-1) + \frac{1}{2 \pi i} \int_{(1/2)} \Psi(-s) E_{k_2-k_1}(z,s)ds .
\end{equation*}
Hence
\begin{align*}
    \inprod{ E_{k_2-k_1}(z | \psi) F_{k_1}}{G_{k_2}}  &= \delta_{f=g} \frac{3}{\pi} \Psi(-1) + \int_{-\infty}^{\infty} \Psi \lr{-\frac{1}{2} - it} \inprod{E_{k_2-k_1} \lr{\cdot, \frac{1}{2}+it} F_{k_1}}{G_{k_2}} dt.
\end{align*}
Now, using Lemma \ref{eisenstion inner products} and the fast decay of $\Psi(s)$ given by \eqref{decay}, we see that
\begin{align*}
     \left | \inprod{ E_{k_2-k_1}(z | \psi) F_{k_1}}{G_{k_2}} - \delta_{f=g} \frac{3}{\pi} \Psi(-1) \right | \ll_{\epsilon} (\log k_2)^{-\frac{1}{12}+ \epsilon} (1+ k_2-k_1)^{1/2}.
\end{align*}
Hence the conclusion follows if 
\begin{equation}
\label{gap}
    k_2-k_1 \leq  \log k_2^{1/6-\epsilon}.
\end{equation}
{\it Case (ii)}: Suppose $L(1, \sym^2 f) L(1, \sym^2 g) \leq (\log k_2)^{-5/6}$. Then we deduce the previous Lemma that $M_{k_1,k_2}(f,g)\ll_{\epsilon} (\log k_2)^{-\frac{1}{24}+ \epsilon
}$. The conclusion follows from Theorem \ref{shifted main} as long as $k_2-k_1 \leq c \log \log k_2$, for some constant $c$. If we optimise out choices, we can let any $c<\frac{1}{12 \log 2 } \asymp 0.12 $.

\subsection{Proof of Theorem \ref{main}}It suffices to to bound $\inprod{u_j R_{k_1}^{k_2} F_{k_1}}{G_{k_2}}$  and $\inprod{E( z | \psi) R_{k_1}^{k_2} F_{k_1}}{G_{k_2}}$.

We begin with the cusp form case. From Lemma \ref{bound cusp}, $\inprod{u_j R_{k_1}^{k_2} F_{k_1}}{G_{k_2}}$ is small when $k_2-k_1 \geq N_{\epsilon}$, for some $N_{\epsilon}$ large enough depending only $\epsilon$. When $k_2-k_1 \leq N_{\epsilon}$, we just combine Lemma \ref{bound cusp} and Lemma \ref{shifted main} depending on the size of $L(1, \sym^2 f) L(1, \sym^2 g)$, as in the previous proof.

For the Eisenstein case, we use that 
\begin{align*}
     \inprod{E (z | \psi) R_{k_1}^{k_2}F_{k_1}}{G_{k_2}} - \delta_{f=g} \frac{3}{\pi} \Psi(-1)  = \int_{-\infty}^{\infty} \Psi \lr{-\frac{1}{2} - it} \inprod{E \lr{\cdot, \frac{1}{2}+it} R_{k_1}^{k_2} F_{k_1}}{G_{k_2}} dt.
\end{align*}
If $k_2-k_1 \leq N_{\epsilon}$, the conclusion follows again easily from Lemma \ref{eisenstion inner products} and the bound  for $\Psi(s)$ on vertical lines given by \eqref{decay} and from Lemma \ref{shifted main}.

If $k_2-k_1$ goes to infinity, we need to obtain a bound for $\Psi(s)s (s+1) \dots (s+n-1)$ in terms of $n$. By repeated partial integration, this boils down to estimating $\| \psi^{(n)} \|_{\infty}$. One problem is that these derivatives can grow arbitrarily fast in terms of $n$. We show that we can work with an approximation $\psi_{\epsilon}$ of $\psi$ such that $ \inprod{E (z | \psi_{\epsilon}) R_{k_1}^{k_2}F_{k_1}}{G_{k_2}}$ is very close to $ \inprod{E (z | \psi) R_{k_1}^{k_2}F_{k_1}}{G_{k_2}}$ and such that we can control $\| \psi_{\epsilon}^{(n)} \|_{\infty}$.

We need to construct a nontrivial function of compact support $\phi$ for which we control the sizes of derivatives $\| \phi^{(n)} \|_{\infty}$, for all $n$. 
From Denjoy--Carleman Theorem \cite[p. 380]{rudin}, we deduce that, for any $\delta>0$, there exists $\phi \in C^{\infty}(\R)$ supported on $[-1, 1]$ such that $\int_{\R} \phi(x) dx = 1$ and $\| \phi^{(n)} \| \ll_{\delta} n^{(1+\delta)n}$, for all $n$. From now on we consider $\delta$ fixed (we will choose it later).

For all $\epsilon>0$, we define $\phi_{\epsilon}(x) = \frac{1}{\epsilon} \phi \lr{\frac{x}{\epsilon}}$ . Then clearly $\phi_{\epsilon}$ is supported on $[-\epsilon, \epsilon]$ and $\int_{\R} \phi_{\epsilon}(x) dx = 1$. Now let any $\psi \in C_b (0,\infty)$. We consider the convolution $$\psi_{\epsilon}(x) : = (\psi * \phi_{\epsilon}) (x)  = \int_{\R} \psi(y) \phi_{\epsilon} (x-y) dy , $$ which is clearly compactly supported in $(0, \infty)$,  for $\epsilon$ small enough. It is not hard to see that
\begin{align*}
    \| \psi- \psi_{\epsilon} \|_{\infty} \leq \epsilon \|\psi'\|_{\infty} .
\end{align*}
Hence, for any $u,v \in \mathcal{L}_k(X)$ such that $\|u \|_2^2 = \|v \|_2^2 = 1$, we have
\begin{align*}
    \left |\inprod{E( z | \psi) u }{v} - \inprod{E( z | \psi_{\epsilon}) u }{v} \right | & = \left |\int_{X} (E( z | \psi) - E( z | \psi_{\epsilon})) u \overline{v} d \mu(z) \right | \\
    & = \left |  \int_0^{\infty} \int_0^1  (\psi(y) - \psi_{\epsilon}(y)) u(x) \overline{v(z)} \frac{dx dy}{y^2}\right | \\
    &\ll_{\psi} \epsilon ,
\end{align*}
since there are $O_{\psi}(1)$ copies of the fundamental domain for which $\psi(y)-\psi_{\epsilon}(y) \neq 0$ and $$\int_{\GH} |u \overline{v}| d \mu(z) \leq \int_{\GH} \frac{|u|^2+|v|^2}{2} d\mu(z) = 1. $$ 
This shows it is enough to consider $ \inprod{E (z | \psi_{\epsilon}) R_{k_1}^{k_2}F_{k_1}}{G_{k_2}}$. Clearly,
\begin{align*}
     \inprod{E (z | \psi_{\epsilon}) R_{k_1}^{k_2}F_{k_1}}{G_{k_2}} &= \delta_{f=g} \frac{3}{\pi} \Psi_{\epsilon}(-1)  + \int_{-\infty}^{\infty} \Psi_{\epsilon} \lr{-\frac{1}{2} - it} \inprod{E \lr{\cdot, \frac{1}{2}+it} R_{k_1}^{k_2} F_{k_1}}{G_{k_2}} dt,
\end{align*}
and we have that $$   \Psi_{\epsilon}(-1) =   \Psi(-1) + O_{\psi}(\epsilon).$$
From definition of $\psi_{\epsilon}$, we have $ \displaystyle \| \psi_{\epsilon}^{(k)} \|_{\infty} \ll_{\psi} \frac{k^{(1+\delta) k}}{\epsilon^k}$. Denote by $\Psi_{\epsilon}$ the Mellin transform of $\psi_{\epsilon}$. From repeated partial integration, we see that for $ |\sigma | \leq 2$, where $s=\sigma+it$, we have
\begin{equation*}
    \Psi_{\epsilon}(s) s (s+1) \dots (s+k-1) \ll_{\psi} \lr{\frac{k C_{\psi}}{\epsilon}}^k k^{\delta k},
\end{equation*}
for some constant $C_{\psi}$ depending on the support of $\psi$.

For simplicity of notation, let $\alpha=\frac{k_2-k_1}{2}$. Choose $\epsilon = \alpha^{-\delta/2}$. We apply Lemma \ref{eisenstion inner products} and choose $k=\alpha +3$.  We have 
\begin{align*}
 A(f,g, \psi_{\epsilon}) &: =   \left | \int_{-\infty}^{\infty} \Psi_{\epsilon} \lr{-\frac{1}{2} - it} \inprod{E \lr{\cdot, \frac{1}{2}+it} R_{k_1}^{k_2} F_{k_1}}{G_{k_2}} dt \right | \\
 &\ll   \frac{\Gamma\lr{\frac{k_1+k_2}{2} }^{1/2} }{\Gamma(k_2)^{1/2} \Gamma (\alpha)^{1/2} } k_2^{1/2} \lr{C_{\psi} \alpha }^{\lr{1+ \frac{3 \delta }{2}} \alpha} .
\end{align*}
Hence
\begin{equation*}
    \log A(f,g, \psi_{\epsilon}) \leq \left ( \frac{1}{2} + \frac{3\delta}{2} \right ) \alpha \log \alpha - \frac{\alpha}{2} \log k_2 + O(\log k_2 + \alpha).
\end{equation*}
The conclusion follows if we pick $\delta=1/12$.

\printbibliography[heading=bibintoc,title={References}]

\end{document}